\documentclass[times,sort&compress,3p]{elsarticle}
\usepackage[labelfont=bf]{caption}

\usepackage{amsmath,amsfonts,amssymb,amsthm,booktabs,color,epsfig,graphicx,hyperref,url}
\usepackage{bm, bbm}

\usepackage{tikz}
\usepackage{tikz-cd}
\usetikzlibrary{trees}
\usetikzlibrary{arrows, positioning}
\usepackage{pgfplots}
\usepackage{pgfplotsthemetol}
\pgfplotsset{compat=1.11}
\usepgfplotslibrary{fillbetween}
\usepgfplotslibrary{groupplots}
\usepackage{subcaption}
\usepackage{csquotes}

\theoremstyle{plain}
\newtheorem{theorem}{Theorem}

\newtheorem{proposition}{Proposition}
\newtheorem{lemma}{Lemma}
\newtheorem{corollary}{Corollary}

\theoremstyle{definition}
\newtheorem{definition}{Definition}
\newtheorem{remark}{Remark}
\newtheorem{example}{Example}

\providecommand{\R}{\ensuremath{\mathbb{R}}}
\providecommand{\C}{\ensuremath{\mathcal{C}}}	
\providecommand{\SI}[1]{\mathcal{S}^{#1}}		

\providecommand{\boldm}[1]{\bm{#1}}						
\providecommand{\rbraces}[1]{\left( #1 \right)} 		
\providecommand{\curly}[1]{\left\{ #1 \right\}} 		
\providecommand{\abs}[1]{\left\lvert #1 \right\rvert} 	
\providecommand{\norm}[1]{\left\lVert #1 \right\rVert} 	
\newcommand{\set}[2]{\left\{ #1 \; \left \arrowvert \; #2 \right. \right\}} 

\providecommand{\TDLn}{\Lambda}									
\providecommand{\TD}[2]{\Lambda \rbraces{#1 \, ; #2}}		
\providecommand{\TDL}[2]{\TD{#1}{#2}}						
\providecommand{\TDLN}[1]{\TDLn \rbraces{#1}}							
\providecommand{\TDLphFo}[1]{\widetilde{#1}}				
\providecommand{\TDLph}[1]{\widetilde{#1}}					
\providecommand{\TDC}{\lambda}								

\providecommand{\loc}{\leq_{loc}}

\providecommand{\tdo}{\leq_{tdo}}
\providecommand{\stdo}{<_{tdo}}
\providecommand{\Inv}[1]{#1^{-1}}							

\providecommand{\wc}{{}\cdot{}}

\begin{document}

\begin{frontmatter}

\title{Multivariate tail dependence and local stochastic dominance}

\author[tud]{Karl Friedrich Siburg}\corref{correspondingauthor}%
\cortext[correspondingauthor]{Corresponding author.}%
\author[tud]{Christopher Strothmann\fnref{scholarship}}%
\fntext[scholarship]{Supported by the German Academic Scholarship Foundation.}
\address[tud]{Faculty of Mathematics, TU Dortmund University, Germany}%

\begin{abstract}
Given two multivariate copulas with corresponding tail dependence functions, we investigate the relation between a natural tail dependence ordering $\tdo$ and the order $\loc$ of local stochastic dominance. 
We show that, although the two orderings are not equivalent in general, they coincide for various important classes of copulas, among them all multivariate Archimedean and bivariate lower extreme value copulas. We illustrate the relevance of our results by an implication to risk management.
\end{abstract}


\end{frontmatter}

\section{Introduction and main results}

It is a fundamental task in statistics and real life evaluation to determine and estimate the dependence between random variables, be they the returns of financial assets in a portfolio or the water levels at various rivers or reservoirs. 
In order to separate the marginal distributions from the dependence structure, the concept of copulas has turned out to be most useful. Sklar's fundamental theorem asserts that the cumulative distribution function $F_{\boldm{X}}$ of some random vector $\boldm{X} = (X_1, \ldots, X_d)$ can be represented by a copula $C$ and its univariate marginal distributions $F_1, \ldots, F_d$ as $$ F_{\boldm{X}}(\boldm{x}) = C(F_1(x_1), \ldots, F_d(x_d)) ~, $$ and vice versa; in addition, the copula $C$ is uniquely defined if the $F_j$ are continuous. This implies that, in order to study dependence properties of continuous random variables, one could just as well study stochastic and/or analytical properties of copulas.

A well-known local type of dependence between random variables, which graphically manifests in the lower left corner of the unit cube $[0, 1]^d$, is the so-called tail dependence. More precisely, for a random vector $\boldm{X}$ with continuous margins $F_1,\ldots,F_d$, the (lower) tail dependence coefficient is defined as
\begin{equation*} 
	\lambda(\boldm{X}) := \lim\limits_{s \searrow 0} \mathbb{P}(X_2 \leq F_{2}^{-1}(s),\ldots,X_d \leq F_{d}^{-1}(s) \mid X_1 \leq F_{1}^{-1}(s))
						= \lim\limits_{s \searrow 0} \frac{C(s,\ldots, s)}{s} =: \lambda(C) ~,
\end{equation*}
whenever this limit exists, and depends only on the underlying copula $C$. $\lambda(\boldm{X})$ was first introduced by \citet{Sibuya.1960} and has since then found wide-spread use in many areas of science, ranging from economics to environmental studies.  
Revisiting the examples from the beginning, $\lambda(\boldm{X})$ describes the probability of extremely large simultaneous losses in a portfolio of stock returns, or the probability of dangerously low water levels at all water reservoirs.
$\lambda(\boldm{X})$ is thus an important indicator for informed investment or policy decisions regarding worst-case scenarios.
Nevertheless, the tail dependence coefficient has some severe drawbacks, which have partly been overcome by using the so-called (lower) tail dependence function 
\begin{equation*} 
	\TDL{\boldm{w}}{\boldm{X}} := \TDL{\boldm{w}}{C} := \lim\limits_{s \searrow 0} \frac{C(s \boldm{w})}{s}
\end{equation*}
with $\boldm{w} \in [0,\infty)^d$ as the natural generalization of $\lambda(\boldm{X}) = \TDL{\boldm{1}}{\boldm{X}}$.

In order to compare the tail dependence of two different copulas $C_1$ and $C_2$, \citet{Koike.2022} and \citet{Siburg.2022a} simultaneously introduced the tail dependence order $\tdo$ by the pointwise ordering of the corresponding tail dependence functions\footnote{To be precise, the tail dependence order in \citet{Koike.2022} and \citet{Siburg.2022a} was defined just for the bivariate case, but the definition trivially extends to the multivariate situation.}: 
$$ C_1 \tdo C_2 :\hspace*{-0.5ex}\iff \TDL{\wc}{C_1} \leq \TDL{\wc}{C_2} ~. $$ 
Having defined the order $\tdo$, it is natural to ask if it is related to other stochastic orders, most notably the ordering by first-order stochastic dominance via the pointwise ordering of the corresponding copulas. 
Since tail dependence is a local concept, however, this consideration makes only sense locally, i.e., for the local stochastic dominance ordering 
$$ C_1 \loc C_2 :\hspace*{-0.5ex}\iff C_1 \leq C_2 \text{ near } 0 ~. $$ 
It follows right from the definitions that $C_1 \loc C_2$ implies $C_1 \tdo C_2$. 
On the other hand, simple examples show that $C_1 \tdo C_2$ does not imply $C_1 \loc C_2$, because there are copulas which are not ordered pointwise near 0 but whose tail dependence functions coincide. 
Therefore, the question remains whether the strict tail dependence ordering implies local stochastic dominance: $$ C_1 \stdo C_2 \implies C_1 \loc C_2 ~? $$ 
Our first main result in this context is the following (see Section~\ref{section:general_tail_behaviour}):

\begin{center}
\textit{While $C_1 \stdo C_2$ does not necessarily imply $C_1 \loc C_2$, it does imply $C_1 \leq C_2$ on any cone near $\boldm{0}$.}
\end{center}

\noindent In spite of this somewhat negative result, we will show in the second part of our paper (see Section~\ref{section:classes_tail_behaviour}) that 

\begin{center}
\textit{$C_1 \stdo C_2$ implies $C_1 \loc C_2$ for various classes of copulas.}
\end{center}

The paper is organized as follows. Section~\ref{section:preliminaries} collects all necessary preliminaries concerning copulas and tail dependence functions. The two stochastic orders $\tdo$ and $\loc$ are introduced in Section~\ref{section:tdo_order}. Section~\ref{section:general_tail_behaviour} contains the discussion and the proof of our first result. The following Section~\ref{section:classes_tail_behaviour} studies the relation between $\tdo$ and $\loc$ for various classes of copulas and, in particular, provides the proofs of the results mentioned in the second claim above. Finally, in Section~\ref{sec:6}, we indicate an implication of our theory to risk management.

\section{Preliminaries} \label{section:preliminaries}

Copulas are central to modern dependence modelling as a means to separate the influence of the marginal distributions from the dependence governing the relationship between a collection of random variables. 
We will present copulas from an axiomatic point of view utilizing $d$-increasing functions, one possible multivariate generalization of monotonicity.
The $d$-increasing property is known in the literature under various terms such as quasi-monotonicity and, in case of $d=2$, supermodularity. 
In the following, we will abbreviate $$ \R_+ := [0, \infty) ~. $$

\begin{definition}
A function $C: [0, 1]^d \rightarrow [0, 1]$ is called a $d$-copula if it fulfils the following properties:
\begin{enumerate}
	\item $C$ is grounded, i.e.\ $C(\boldm{u}) = 0$ if $u_k = 0$ for some $k$.
	\item $C$ has uniform margins, i.e.\ $C(\boldm{u}) = u_k$ if $u_\ell=1$ for all $\ell \neq k$.  
	\item $C$ is $d$-increasing. 
\end{enumerate}
\end{definition}

The so-called tail dependence quantifies extremal dependence between multiple random variables. 
Introduced by \citet{Sibuya.1960}, the (lower) tail dependence coefficient of a copula $C$ is defined as 
\begin{equation*} 
	\lambda(C) = \lim\limits_{s \searrow 0} \frac{C(s, \ldots, s)}{s} = \lim\limits_{s \searrow 0} \frac{C(s \boldm{1})}{s} ~, 
\end{equation*}
with $\boldm{1} := (1, \ldots, 1)$ whenever this limit exists. We have the following natural generalization of $\lambda$; see \citet{Joe.1997}.

\begin{definition}  \label{def:tail_dependence_function}
For a $d$-copula $C$, the (lower) tail dependence function $\TDLn: \R_+^d \rightarrow \R_+$ is defined as
\begin{gather*}
		\TDLN{\boldm{w}} := \TDL{\boldm{w}}{C} 	:= \lim\limits_{s \searrow 0} \frac{C(s \boldm{w})}{s}
\end{gather*}
provided that the limit exists for all $\boldm{w} \in \R_+^d$.
\end{definition}

The next theorem reveals how the tail dependence function influences the behaviour of $C$ near 0.

\begin{theorem} \label{thm:concave_quasi_copula}
Suppose $C$ is a $d$-copula and $L: \R_+^d \rightarrow \R_+$ is positive homogeneous of order $1$, i.e.\ $L(s \boldm{w}) = s L(\boldm{w})$ holds for all $s > 0$ and $\boldm{w} \in \R_+^d$. 
Then the following are equivalent:
\begin{enumerate}
    \item The tail dependence function $\TDL{\boldm{w}}{C}$ exists for all $\boldm{w} \in \R_+^d$ and equals $L(\boldm{w})$. 
    \item $L$ is the leading part of a uniform lower tail expansion of $C$, i.e.
    \begin{equation*}
        C(\boldm{u}) = L(\boldm{u}) + R(\boldm{u}) \norm{\boldm{u}}_1 = L(\boldm{u}) + R(\boldm{u})\rbraces{u_1 + \ldots + u_d} ~,
    \end{equation*} 
    where $R: [0,1]^d \rightarrow \R$ is a bounded function fulfilling $R(\boldm{u}) \rightarrow 0$ as $\norm{\boldm{u}}_1 \rightarrow 0$. 
\end{enumerate}
In addition, the tail dependence function $\TDLn$ is concave. 
\end{theorem}

\begin{proof}
A proof can be found in \citet{Jaworski.2006} or \citet{Jaworski.2010b}.
\end{proof}

\begin{remark}
While, apparently, not discussed in the literature, Theorem~\ref{thm:concave_quasi_copula} can actually be extended to hold for quasi-$d$-copulas as well. The first part relies on the theory of B-differentiability discussed in \citet{Scholtes.2012}, and the concavity follows from Theorem~2.10 in \citet{Konig.2003} whenever the quasi-$d$-copula is supermodular in a neighbourhood of $\boldm{0}$.
\end{remark}

Theorem~\ref{thm:concave_quasi_copula} shows that the tail dependence function $\TDL{\wc}{C}$ is the directional derivative of $C$ in $\boldm{0}$.
We point out, however, that $C$ is (Fréchet-) differentiable in $\boldm{0}$ if and only if $\TDL{\wc}{C}$ vanishes identically.

\begin{example} \label{ex:gaussian_tdf}
Suppose $C_{\rho}$ is a Gaussian copula with correlation coefficient $\rho \in [-1, 1]$. 
Then the tail dependence coefficient of $C_{\rho}$ vanishes for every $\rho < 1$, and jumps to 1 when $\rho=1$. 
This makes Gaussian models often unsuitable (and sometimes even dangerous) in the context of financial assets since extreme simultaneous losses are severely underestimated.
\end{example}

Before investigating how the tail dependence function can be used to quantify the extremal behaviour of the underlying random vector, we briefly present some theoretical properties of $\TDLn$. 
Many of these properties are directly linked to the properties of the corresponding copula $C$ and the fact that $\TDLn$ is the directional derivative of $C$ in $\boldm{0}$.

\begin{proposition} \label{prop:tail_dep_func}
A function $\TDLn: \R_+^d \rightarrow \R_+$ is the tail dependence function of a $d$-copula $C$ if and only if
\begin{enumerate}
	\item[a.] $\TDLn$ is bounded from below by $0$ and from above by $\TDLn^+ := \TDL{\wc}{C^+}$ \label{prop:tail_dep_func_bounded}
	\item[b.] $\TDLn$ is $d$-increasing 
	\item[c.] $\TDLn$ is positive homogeneous of order $1$.\label{prop:tail_dep_func_homogeneous}
\end{enumerate}  
Furthermore, every tail dependence function $\TDLn(\wc)$ is Lipschitz continuous; more precisely,
  \begin{equation*}
  	\abs{\TDLn(\boldm{v})	- \TDLn(\boldm{w})}	\leq \sum\limits_{k = 1}^d \abs{v_k - w_k}	~.
  \end{equation*} 
\end{proposition}

\begin{proof}
For a proof, see \citet{Jaworski.2006} or \citet{Jaworski.2010b}.
\end{proof}

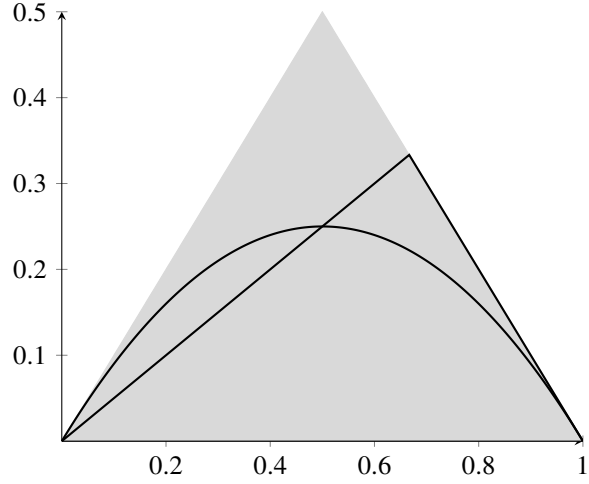
\begin{figure}
\centering
\begin{tikzpicture}[ declare function={ 
    			tdf(\t) 	= min(\t, 1-\t); 
				tdf1(\t)	= \t*(1-\t);
				tdf2(\t) 	= min(0.5*\t, 1-\t);
			}]
        \begin{axis}[axis lines=middle]
            \addplot[name path=f, domain=0:1, samples=151, color=gray!30]{tdf(\x)};
            \addplot[name path=g, domain=0:1, samples=151, color=black, thick]{tdf1(\x)};
            \addplot[name path=h, domain=0:1, samples=151, color=black, thick]{tdf2(\x)};
			\path[draw,name path = xAxis] (axis cs:0,0) -- (axis cs:1,0);
            \addplot[fill=gray!30] fill between [of=xAxis and f]; 
        \end{axis}
\end{tikzpicture} 
\caption{Example of two concave functions (depicted in black) lying in the area between $0$ and $\min \curly{t, 1-t}$ (depicted in grey), thus inducing corresponding bivariate tail dependence functions.}
\label{fig:valid_tdf}
\end{figure}

\begin{remark} \label{ex:def_ev}
Due to the positive homogeneity of the tail dependence function $\TDLn$, we can reduce the domain of $\TDLn$ from $\R_+^d$ to the simplex $\mathcal{S}^{d-1} := \set{\boldm{w} \in \R_+^d}{\norm{\boldm{w}}_1 = 1}$ with no loss of information. In particular, for dimension $d=2$, we may consider $\TDLn$ on the unit simplex $$ \SI{1} := \set{\boldm{w} \in \R_+^2}{\boldm{w} = (t, 1-t) \text{ with } t \geq 0 } $$ and identify $\SI{1}$ with $[0, 1]$ such that $\TDLphFo{\Lambda}(t) := \TDLn(t, 1-t)$
is a univariate function $\TDLph{\Lambda}: [0, 1] \rightarrow [0, 1/2]$ with 
\begin{gather}\label{eqn:bivariate_tdf_characterization}
\begin{aligned} 
	\TDLn(\boldm{w})	= \norm{\boldm{w}}_1 \TDLn \rbraces{\frac{\boldm{w}}{\norm{\boldm{w}}_1}} 
						= (w_1 + w_2) \TDLn \rbraces{\frac{w_1}{w_1 + w_2}, 1-\frac{w_1}{w_1 + w_2}} 
						= (w_1 + w_2) \TDLph{\Lambda} \rbraces{\frac{w_1}{w_1 + w_2}} ~.
\end{aligned}
\end{gather}
From \citet{Gudendorf.2010}, it then follows that $\TDLn$ is a bivariate tail dependence function if and only if  $\TDLph{\Lambda}: [0, 1] \rightarrow [0, 1/2]$ is a concave function fulfilling $0 \leq \TDLph{\Lambda}(t) \leq \min \curly {t, 1-t}$.
In other words, any concave function whose graph lies in the grey area depicted in Figure~\ref{fig:valid_tdf} induces a bivariate tail dependence function via \eqref{eqn:bivariate_tdf_characterization}.
\end{remark}

\begin{example}
For a bivariate copula $C$, the tail dependence coefficient is given by
\begin{equation} \label{eq:tdc_restricted}
	\lambda(C) = 2 \TDLph{\Lambda}\rbraces{\frac{1}{2} \, ; C} .
\end{equation}
\end{example}

\section{Tail dependence order \texorpdfstring{$\tdo$}{} and local stochastic dominance \texorpdfstring{$\loc$}{}} \label{section:tdo_order}

The following natural order of tail dependence was introduced simultaneously in \citet{Koike.2022} and \citet{Siburg.2022a} for the bivariate setting.

\begin{definition}
We say that a $d$-copula $C_1$ is less tail dependent than another $d$-copula $C_2$, written $C_1 \tdo C_2$, if and only if $$\TD{\boldm{w}}{C_1} \leq \TD{\boldm{w}}{C_2} $$ for all $\boldm{w} \in \R^d_+$.
Similarly, $C_1$ is strictly less tail dependent than $C_2$, written $C_1 \stdo \C_2$, if and only if $$\TD{\boldm{w}}{C_1} < \TD{\boldm{w}}{C_2} $$ holds for all $\boldm{w} \in (0, \infty)^d$.
\end{definition}

We will see immediately that $\tdo$ is not an order but only a preorder; following common practice, however, we will nevertheless call $\tdo$ the \emph{tail dependence order}.

\begin{proposition}
The tail dependence order $\tdo$ is a preorder, i.e., it is reflexive and transitive; however, it is neither antisymmetric nor total. 
Furthermore,
\begin{enumerate}
	\item the greatest and maximal elements of this preorder are copulas with a tail dependence function equal to $\TD{\boldm{w}}{C^+}$;
	\item the least and minimal elements are copulas with vanishing tail dependence function.
\end{enumerate}
\end{proposition}

\begin{proof}
The first assertion follows immediately from the properties of the pointwise ordering on the space of continuous functions.
Furthermore, $\tdo$ is not antisymmetric due to $\TD{\boldm{w}}{\Pi} = \TD{\boldm{w}}{C^-}$ and $\Pi \neq C^-$.
To see that $\tdo$ is not total, consider as an example the non-ordered tail dependence functions $\min \curly{2w_1, w_2, \ldots, w_d}/2$ and $\min \curly{w_1, w_2, \ldots, 2w_d}/2$ constructed, e.g., via the gluing technique from \citet{Siburg.2008b}.
To obtain the least and greatest elements, we apply Proposition~\ref{prop:tail_dep_func} stating that 
$$0 \leq \TD{\boldm{w}}{C} \leq \TD{\boldm{w}}{C^+}$$ holds for all $d$-copulas $C$. 
As these bounds are sharp, every copula $C$ with $\TD{\boldm{w}}{C}=0$ or $\TD{\boldm{w}}{C}=\TD{\boldm{w}}{C^+}$ is a least or greatest element, respectively. 
If $C$ is a maximal element then 
\begin{equation*}
	C \tdo D \implies D \tdo C
\end{equation*}
for any copula $D$. Thus, for $D = C^+$, we have $\TD{\boldm{w}}{C} = \TD{\boldm{w}}{C^+}$. 
Analogously, every minimal element must have vanishing tail dependence function.
\end{proof}

\begin{remark} \label{rem:too_tdo_equivalence}
An order similar to $\tdo$ was introduced in \citet{Li.2013} by ordering the copula values along rays.
Namely, $C_1$ is smaller than $C_2$ in the so-called tail orthant order, in short $C_1 \leq_{too} C_2$, if for all $\boldm{w} \in \R_+^d$ there exists a $t_{\boldm{w}} > 0$ such that $C_1(s\boldm{w}) \leq C_2(s\boldm{w})$ holds for all $s \leq t_{\boldm{w}}$.
While $\leq_{too}$ implies $\tdo$, the converse does not hold.

To see this, consider the Archimedean Joe-copula (family~(6) in \citet{Charpentier.2009}) with parameter $\theta = 2$, which we will call $C$ for the sake of simplicity.
Gluing (see \citet{Siburg.2008b}) $C$ and $C^+$ with respect to the first and second component results in the copulas $C_1$ and $C_2$, respectively. 
Now, on the one hand, as $C$ has vanishing tail dependence function, so have $C_1$ and $C_2$, and $\TDL{\wc}{C_1} = \TDL{\wc}{C_2}$.
On the other hand, $C_1$ and $C_2$ are strictly ordered conversely along the directions $\boldm{w}_1 = (1/2, 1)$ and $\boldm{w}_2 = (1, 1/2)$, so neither $C_1 \leq_{too} C_2$ nor $C_2 \leq_{too} C_1$ can hold.
\end{remark}

The following result shows that the tail dependence order implies the ordering of the tail dependence coefficient but not vice versa.

\begin{proposition} \label{prop:tdc_monotone}
$C_1 \tdo C_2$ implies $\TDC(C_1) \leq \TDC(C_2)$, but the converse is not true in general.
\end{proposition}

\begin{proof}
The first assertion is trivial in view of $\TDC(C) = \TDL{\boldm{1}}{C}$. 
The second assertion follows from \eqref{eq:tdc_restricted} and the counterexamples in Figure~\ref{fig:valid_tdf} where the two black curves coincide at the point $1/2$.
\end{proof}

\begin{definition}
Let $C_1$ and $C_2$ be $d$-copulas.
\begin{enumerate}
	\item $C_1$ is called smaller than $C_2$ in the lower orthant order, abbreviated $C_1 \leq C_2$, if $C_1(\boldm{u}) \leq C_2(\boldm{u})$ for all $\boldm{u} \in [0, 1]^d$.
	\item $C_1$ is called smaller than $C_2$ in the local lower orthant order, abbreviated $C_1 \loc C_2$, if there exists an $\varepsilon > 0$ such that $C_1(\boldm{u}) \leq C_2(\boldm{u})$ for all $\boldm{u} \in B_\varepsilon(\boldm{0}) \cap [0, 1]^d$.
\end{enumerate}
\end{definition}

Note that $C_1 \leq C_2$ means that, unless $C_1=C_2$, $C_1$ (first-order) stochastically dominates $C_2$. 
Analogously, we say that, if $C_1 \loc C_2$ with $C_1 \neq C_2$, $C_1$ locally (first-order) stochastically dominates $C_2$. Note further that $\loc$ is strictly weaker than $\leq$ which is easily seen using gluing techniques as in \citet{Siburg.2008b}. 

A straightforward connection between $\tdo$ and $\loc$ is obtained by noting that smaller values of $C$ result in smaller values of $\Lambda$, leading to the following simple observation.

\begin{proposition}
$$ C_1 \loc C_2 \implies C_1 \tdo C_2 ~. $$ 
\end{proposition}

Intuitively, from the representation $C(\boldm{u}) = \TD{\boldm{u}}{C} + R(\boldm{u})\rbraces{u_1 + \ldots + u_d}$ in Theorem~\ref{thm:concave_quasi_copula} one might be tempted to expect also the converse implication $C_1 \tdo C_2 \implies C_1 \loc C_2$. This, however, does not hold in general because any two non-ordered copulas with vanishing tail dependence functions yield a counterexample; see Remark~\ref{rem:too_tdo_equivalence}. Therefore, the question remains whether
\begin{equation} \label{eqn:core_result}
	C_1 \stdo C_2 \implies C_1 \loc C_2  ~?
\end{equation}
This question will be investigated in detail in the rest of the paper.

\section{The implication \texorpdfstring{$C_1 \stdo C_2 \implies C_1 \loc C_2$}{(5)} does not hold in general, but only on cones} \label{section:general_tail_behaviour}

The following example, communicated to us by P. Jaworski, shows that \eqref{eqn:core_result} does not hold in general.

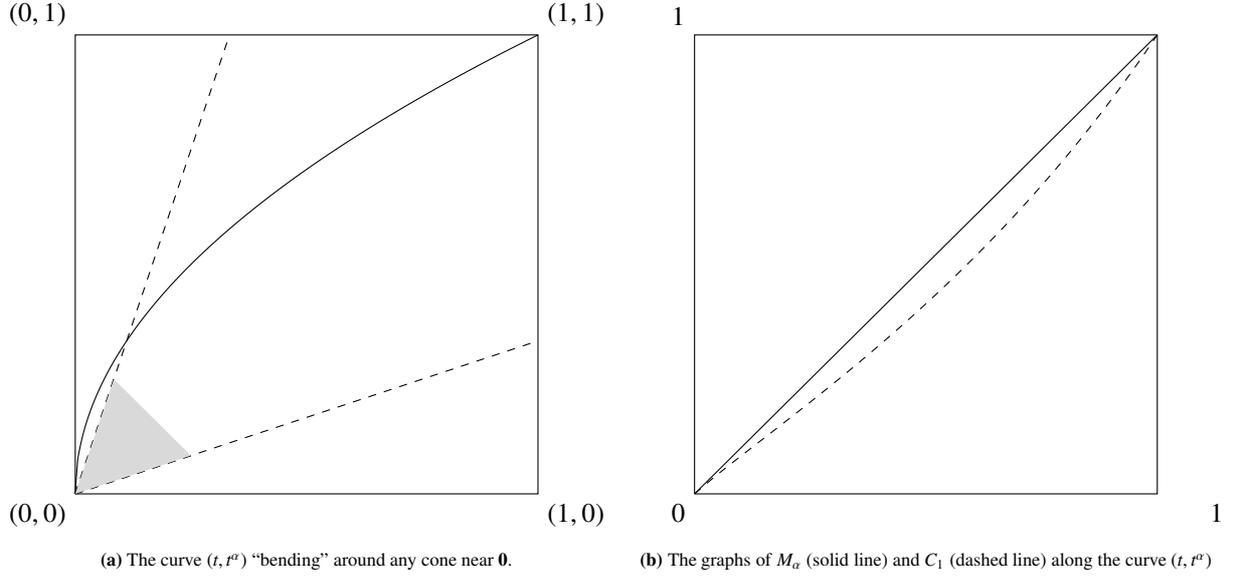
\begin{figure}
\begin{subfigure}[b]{0.49\textwidth}
    \centering
    \resizebox{\linewidth}{!}{\begin{tikzpicture}
	\draw (0,0) node[anchor=north east] {$(0,0)$} -- (6,0) node[anchor=north west] {$(1,0)$} -- (6,6)  node[anchor=south west] {$(1,1)$} -- (0,6)  node[anchor=south east] {$(0,1)$} -- (0,0);
    \draw[domain=0:6, samples=200, smooth, variable=\x] plot ({\x}, {sqrt(6)*6*sqrt{(\x)/6}});
    \draw[domain=0:6, smooth, variable=\x, dashed] plot({\x}, {\x/3});
    \draw[name path = A, domain=0:1.5, smooth, variable=\x, dashed] plot({\x}, {\x/3});
    \draw[domain=0:6, smooth, variable=\x, dashed] plot({\x/3}, {\x});
    \draw[name path = B, domain=0:1.5, smooth, variable=\x, dashed] plot({\x/3}, {\x});
    \tikzfillbetween[of=A and B]{fill=gray!30};
\end{tikzpicture}}
\subcaption{The curve $(t,t^\alpha)$ \enquote{bending} around any cone near $\boldm{0}$.}
\label{fig:mo-clayton-a}
\end{subfigure}
\begin{subfigure}[b]{0.49\textwidth} 
    \centering
    \resizebox{\linewidth}{!}{\begin{tikzpicture}
	\draw (0,0) node[anchor=north east] {\phantom{$(0,)$}$0$} -- (6,0) node[anchor=north west] {\phantom{$(1,)$}$1$} -- (6,6)  node[anchor=south west] {} -- (0,6)  node[anchor=south east] {$1$} -- (0,0);
    \draw[domain=0:6, samples=200, smooth, variable=\x] plot ({\x}, {\x});
    \draw[domain=0:1, smooth, variable=\x, dashed] plot({6*\x}, {6*\x/(1 + sqrt(\x) - \x)});
\end{tikzpicture}}
\subcaption{The graphs of $M_\alpha$ (solid line) and $C_1$ (dashed line) along the curve $(t, t^\alpha)$}
\label{fig:mo-clayton-b}
\end{subfigure}
\caption{Plots illustrating Example~\ref{example:eo_doesnotimply_loo} for $\alpha = \frac{1}{2}$.}
\label{fig:mo-clayton-example}
\end{figure}

\begin{example} \label{example:eo_doesnotimply_loo}
Keeping in mind the ray-like definition of the tail dependence function, we use the curve $(t, t^\alpha)$ that "bends" around any given cone near $\bm{0}$; see Figure~\ref{fig:mo-clayton-a}. 
Consider the Marshall-Olkin copula $M_\alpha(\boldm{u}) = \min \curly{u_1^{1-\alpha} u_2, u_1}$ with parameter $\alpha \in (0, 1)$ and the Archimedean Clayton copula $C_\vartheta$ with $\vartheta > 0$ given by 
$C_\vartheta (\boldm{u}) = \big(u_1^{-\vartheta} + u_2^{-\vartheta} - 1\big)^{-1/\vartheta}$.
Note that the singular support of $M_{\alpha}$ is precisely the curve $(t, t^\alpha)$.

While $M_\alpha$ is strictly smaller than $C_\vartheta$ in the tail dependence order in view of
\begin{equation*}
	\TD{\boldm{w}}{M_\alpha} = 0 < \TD{\boldm{w}}{C_\vartheta} \text{ for all } \boldm{w} \in (0, \infty)^2 ~,
\end{equation*}
setting $\boldm{u} := (t, t^\alpha)$ yields the opposite inequality for the copulas themselves, namely
\begin{equation*}
	C_\vartheta(t, t^\alpha) 	= \frac{t^{1+\alpha}}{\rbraces{t^\vartheta + t^{\alpha \vartheta} - t^{(1+\alpha)\vartheta}}^{\frac{1}{\vartheta}}}
								= \frac{t}{\rbraces{t^{(1-\alpha)\vartheta} + 1 - t^{\vartheta}}^{\frac{1}{\vartheta}}}
								< t = M_\alpha(t, t^\alpha) ~;
\end{equation*}
see Figure~\ref{fig:mo-clayton-b}. Thus, even the strict tail dependence ordering does not imply an ordering of the values of the underlying copulas.
\end{example}

The next result shows that the reasoning in the above example is the only possible in the sense that, as soon as we move away from the coordinate axes, the implication in \eqref{eqn:core_result} becomes true.

\begin{theorem} \label{thm:eo_implies_cone_loo}
Suppose the $d$-copulas $C_1$ and $C_2$ fulfil $C_1 \stdo C_2$.
Then for any cone\footnote{A set $S \subset \mathbb{R}^d$ is called a cone if $\boldm{w} \in S$ implies $\lambda \boldm{w} \in S$ for every $\lambda > 0$.} $S \subset (0, \infty)^d$ such that $S \cup \curly{\boldm{0}}$ is closed in $[0, \infty)^d$, there exists an $\varepsilon > 0$ such that $C_1 (\boldm{u}) \leq C_2 (\boldm{u})$ for all $\boldm{u} \in S \cap B_\varepsilon (\boldm{0})$.
\end{theorem}

\begin{proof}
Our proof follows a standard technique in convex analysis (see, e.g., \citet{Scholtes.2012}). 
For ease of notation, define $f: (S \cup \curly{\boldm{0}}) \cap [0, 1]^d \rightarrow [-1, 1]$ as $$ f(\boldm{u}) := C_2(\boldm{u}) - C_1(\boldm{u}) ~. $$ Due to our assumption $C_1 \stdo C_2$, the directional derivative of $f$ in $\boldm{0}$ exists and is strictly positive, i.e.\ $f'(\boldm{0} \, ; \, \boldm{s}) > 0$ for all $\boldm{s} \in S$. Moreover, every copula is Lipschitz continuous with Lipschitz constant $1$, so $f$ is Lipschitz continuous with Lipschitz constant $2$.

We will show that $\boldm{0}$ is a local minimum of $f$ which in turn implies
\begin{equation*}
	C_2 (\boldm{u}) - C_1(\boldm{u}) = f(\boldm{u})	\geq 0 
\end{equation*} 
in a neighbourhood of $\boldm{0}$. 
To do so, we assume by contradiction that there exists a sequence $\boldm{w}_n \in S$ with $\boldm{w}_n \rightarrow \boldm{0}$ and $f(\boldm{w}_n) \leq f(\boldm{0})$. 
Decompose $\boldm{w}_n = r_n \boldm{s}_n$ with $r_n := \norm{\boldm{w}_n}_1$ and $\boldm{s}_n := \boldm{w}_n/\norm{\boldm{w}_n}_1$. 
Then $\boldm{s}_n$ has a convergent subsequence, again denoted by $\boldm{s}_n$, with $\boldm{s}_n \rightarrow \boldm{s}^* \in S$, and it holds that
\begin{align*}
	0	&\geq \frac{f(\boldm{w}_n) - f(\boldm{0})}{r_n}	
		= \frac{f(r_n \boldm{s}_n) - f(\boldm{0})}{r_n} \\
		&= \frac{f(r_n \boldm{s}_n) - f(r_n \boldm{s}^*)}{r_n} + \frac{f(r_n \boldm{s}^*) - f(\boldm{0})}{r_n}
		\rightarrow f'(\boldm{0} \ ; \ \boldm{s}^*)~,
\end{align*}
where the convergence of the first summand follows from the abovementioned Lipschitz continuity of $f$ since
\begin{equation*}
	0	\leq \frac{\abs{f(r_n \boldm{s}_n) - f(r_n \boldm{s}^*)}}{r_n}
		\leq \frac{2 \norm{r_n \boldm{s}_n - r_n \boldm{s}^*}_1}{r_n}
		= 2\norm{\boldm{s}_n - \boldm{s}^*}_1
		\rightarrow 0 ~.
\end{equation*}
But $f'(\boldm{0} \ ; \ \boldm{s}^*) > 0$ which leads us to the desired contradiction. \qedhere
\end{proof}

\begin{remark}
We point out that all notions and results of Sections~\ref{section:tdo_order} and \ref{section:general_tail_behaviour} can also be formulated and proven around any other vertex of $[0,1]^d$ instead of $\bm{0}$, by permutating the vertices accordingly; see \citet[Sect.~1.7.3]{Durante.2015}.   
\end{remark}

\section{When the implication \texorpdfstring{$C_1 \stdo C_2 \implies C_1 \loc C_2$}{(5)} does hold} \label{section:classes_tail_behaviour}

We will now focus on various copula families for which the implication in \eqref{eqn:core_result} holds true, thus yielding a description of the (strict) tail dependence order in terms of local stochastic dominance. 
We know from Section~\ref{section:general_tail_behaviour} that \eqref{eqn:core_result} does not hold in general, due to possible irregular behaviour of the copulas near the coordinate axes.

If we restrict ourselves to specific classes of copulas, however, thereby providing some additional regularity, we may be able to control the limiting behaviour near the axes and prove the implication in \eqref{eqn:core_result}.
All the example classes in this section share a common feature---they can be described by one single (oftentimes univariate) generating function. 
The idea to prove the implication $$ C_1 \stdo C_2 \implies C_1 \loc C_2 $$ is then to prove that $C_1 \stdo C_2$ implies an ordering of the generating functions near 0, allowing us to conclude that $C_1 \loc C_2$.

\subsection{Copulas with a prescribed diagonal}

Given a $d$-copula $C$, the function $\delta: [0, 1] \rightarrow [0, 1]$ with $\delta(t) := C(t\boldm{1})$ is called its diagonal section. The question arises which functions $\delta$ can be realized as diagonal sections.

\begin{proposition} \label{def:diagonal_copula}
A function $\delta: [0, 1] \rightarrow [0, 1]$ is the diagonal of some $d$-copula if, and only if, it satisfies the following properties:
\begin{enumerate}
	\item $\delta(0) = 0$ and $\delta(1) = 1$
    \item $\delta(t) \leq t$
	\item $\delta$ is increasing
	\item $\abs{\delta(s) - \delta(t)} \leq d \abs{s-t}$.
\end{enumerate} 
\end{proposition}%

For a proof, we refer to \citet[Thm.~2.6.1]{Durante.2015}.

\begin{example}
In the bivariate case, there are two well-known constructions of copulas with a prescibed diagonal $\delta$. First of all, the function
\begin{equation} \label{eq:Fredricks_Nelsen}
	C_{FN}(u_1,u_2) := \min \curly{u_1, u_2, \frac{\delta(u_1) + \delta(u_2)}{2}}
\end{equation}
is a 2-copula with diagonal section $\delta$ and called \emph{Fredricks-Nelsen copula}. It can be shown that the Fredricks-Nelsen copula is the pointwise maximal copula among all symmetric\footnote{A $2$-copula $C$ is called symmetric if $C(u_1,u_2) = C(u_2,u_1)$.} copulas with the same prescribed diagonal.

Another copula with a prescribed diagonal $\delta$ is given by
\begin{equation} \label{eq:Bertino}
    C_B(u_1, u_2) := \begin{cases} u_1 - \min\limits_{t \in [u_1, u_2]} (t - \delta(t)) &, \, u_1 \leq u_2 \\ u_2 - \min\limits_{t \in [u_2, u_1]} (t - \delta(t)) &, \, u_2 \leq u_1 \end{cases}~.
\end{equation}
This is the so-called \emph{Bertino copula}, which is the pointwise minimum among all copulas with the same diagonal $\delta$.
\end{example}

The following result shows that, in any dimension, the strict tail dependence ordering implies the local ordering of the corresponding diagonals near $0$.

\begin{theorem} \label{lemma:diagonal_copulas}
    Let $C_1$ and $C_2$ be two $d$-copulas, each admitting a tail dependence function, with diagonal sections $\delta_1$ and $\delta_2$ such that $C_1 \stdo C_2$. Then there is an $\epsilon > 0$ such that $$ \delta_1(t) \leq \delta_2(t) $$ for all $t\in [0,\epsilon]$.
\end{theorem}

\begin{proof}
    This follows immediately from
    \begin{equation*}
	\lim\limits_{t \searrow 0} \frac{\delta_2(t) - \delta_1(t)}{t}	= \lim\limits_{t \searrow 0} \frac{C_2(t\boldm{1}) - C_1(t\boldm{1})}{t}
																	= \TD{\boldm{1}}{C_2} - \TD{\boldm{1}}{C_1} 
																	> 0  ~.
\end{equation*}
\end{proof}

\begin{corollary} \label{cor:FN_Bertino}
For the classes of bivariate Fredricks-Nelsen, respectively Bertino, copulas we have $$ C_1 \stdo C_2 \implies C_1 \loc C_2 ~. $$ 
\end{corollary}

\begin{proof}
    Since $C_1 \stdo C_2$ implies that $\delta_1(t) \leq \delta_2(t)$ for small enough $t$ in view of Theorem~\ref{lemma:diagonal_copulas}, the result follows from the explicit representations in \eqref{eq:Fredricks_Nelsen} and \eqref{eq:Bertino}.
\end{proof}

The following special class of symmetric bivariate copulas called \emph{semilinear copulas} was introduced in \citet{Durante.2006}, and later characterized in terms of the diagonal function alone in \citet[Thm.~4]{Durante.2008b}.

\begin{proposition}
Let $\delta: [0, 1] \rightarrow [0, 1]$ be the diagonal of some $2$-copula.
Then the function
\begin{equation} \label{eq:semilinear_copula}
    C_{SL}(u_1, u_2) := \min(u_1, u_2) \cdot \frac{\delta(\max(u_1, u_2))}{\max(u_1, u_2)}
\end{equation}
is a $2$-copula if, and only if, $\delta$ satisfies the following properties:
\begin{enumerate}
	\item $\frac{\delta(t)}{t}$ is increasing on $(0,1]$
	\item $\frac{\delta(t)}{t^2}$ is decreasing on $(0,1]$.
\end{enumerate} 
\end{proposition}

By construction, the semilinear copula $C_{SL}$ has the function $\delta$ as its the diagonal, and its tail dependence function is given by
$$ \TD{(w_1, w_2)}{C_{SL}}    = \min(w_1, w_2) \cdot \lim\limits_{s\searrow 0} \frac{\delta(s \max(w_1, w_2))}{s\max(w_1, w_2)} 
                                = \min(w_1, w_2) \cdot \lim\limits_{t\searrow 0} \frac{\delta(t)}{t}  $$ 
where the limit exists by the assumptions on $\delta$. Therefore, as in the proof of Corollary~\ref{cor:FN_Bertino}, we conclude from Theorem~\ref{lemma:diagonal_copulas} the following result for semilinear copulas.

\begin{corollary}
For the class of semilinear copulas we have $$ C_1 \stdo C_2 \implies C_1 \loc C_2 ~. $$ 
\end{corollary}

\subsection{Extreme value copulas}

The next class of copulas we consider are the so-called extreme value copulas. We will see that, in the bivariate case, they are uniquely characterized by their tail dependence function. This allows us to construct copulas with a given tail dependence function.

\begin{definition} 
Suppose $\TDLn$ is a bivariate tail dependence function.
Then 
\begin{equation} \label{eqn:extreme_value_copula_def}
	C^{EV}(u_1, u_2 \, ; \TDLn)	:= \exp \rbraces{\log(u_1) + \log(u_2) + \Lambda(-\log(u_1), -\log(u_2))} 
\end{equation}
is a $2$-copula, called an extreme value copula. The survival copula $\widehat{C^{EV}}$ of $C^{EV}$, defined as
\begin{equation*}
	C^{LEV}(u_1, u_2 \, ; \TDLn) := \widehat{C^{EV}}(u_1, u_2 \, ; \TDLn) := u_1 + u_2 - 1 + C^{EV}(1-u_1, 1-u_2)	~,
\end{equation*}
is called lower extreme value copula. 
\end{definition}

That $C^{EV}$ is indeed a copula is shown, e.g., in \citet[Thm.~6.6.7]{Durante.2015}. The next proposition states that the lower extreme value copula has $\Lambda$ as its tail dependence function; see also \citet{Durante.2015}.

\begin{proposition} \label{prop:extreme_value_copula}
If $\TDLn$ is a bivariate tail dependence function and $C^{LEV}$ its corresponding lower extreme value copula then
\begin{equation*}
	\TDL{\cdot}{C^{LEV}}	= \TDLn(\cdot)	~.
\end{equation*}
\end{proposition}

\begin{remark}
While the construction of a $d$-copula with a given tail dependence function is quite straightforward in dimension $d=2$, whereas the general construction for $d > 2$ is more involved and can be found in \citet[Prop.~6]{Jaworski.2006}.
\end{remark}

\begin{theorem} \label{prop:evc_oo}
The tail dependence order is equivalent to the lower orthant order for lower extreme value $2$-copulas.
That is, 
\begin{equation*}
	\TD{\boldm{w}}{C_1}	\leq \TD{\boldm{w}}{C_2} \Leftrightarrow \widehat{C}_1 \leq \widehat{C}_2 \Leftrightarrow C_1 \leq C_2	\Leftrightarrow C_1 \loc C_2	
\end{equation*}
holds for all lower extreme-value $2$-copulas $C_1$ and $C_2$, where $\widehat{C}_1$ denotes the survival copula of $C_1$. 
\end{theorem}

\begin{proof}
The proof follows immediately from \eqref{eqn:extreme_value_copula_def}, Proposition~\ref{prop:extreme_value_copula} and the fact that, for all $2$-copulas, $C_1 \leq C_2$ is equivalent to $\widehat{C}_1 \leq \widehat{C}_2$.
\end{proof}

\subsection{Archimedean copulas}

Another important class of copulas are the Archimedean copulas which are determined by a single univariate function, the so-called generator. Archimedean copulas have become a popular tool both for theoretical as well as practical considerations. 
Their simple form allows for the explicit calculation of many dependence measures, for instance, thus enabling a precise understanding of their induced dependencies.
Empirical studies have suggested that many pairs of stocks can be modelled quite successfully by Archimedean copulas; see, e.g., \citet{Buecher.2012} and \citet{Trede.2013}.

\begin{definition}
We call a function $\phi: [0, 1] \rightarrow [0, \infty]$ a generator if it is continuous, strictly decreasing and fulfils $\phi(1) = 0$.
Furthermore, we call $\phi$ strict if 
\begin{equation*}
	\lim\limits_{s \searrow 0 } \phi(s) = \infty ~.
\end{equation*}
\end{definition}

\begin{definition} \label{def:archimedean_copula}
A $d$-copula $C$ is called Archimedean if there exists a generator $\phi$ such that
\begin{equation} \label{def:arch_copula}
	C(\boldm{u}) = \phi^{[-1]} \rbraces{\sum\limits_{k = 1}^d \phi(u_k)}
\end{equation}
where $\phi^{[-1]}(x) := \inf \set{t \in [0, 1]}{\phi(t) \leq x}$ denotes the generalized inverse of $\phi$. 
\end{definition}

To obtain a closed form for the tail dependence function of an Archimedean copula, its generator $\phi$ is often assumed to be regularly varying. 
This condition poses in practice only a slight restriction since it is fulfilled by virtually all relevant generators as demonstrated by \citet{Charpentier.2009}.
For an extensive treatment of regularly varying functions and related topics we refer to \citet{Bingham.1987}. 

\begin{definition} \label{def:reg_var}
A positive measurable function $f$ on $\R_+$ is called regularly varying at $\infty$ with index $\alpha \in \R$ if
\begin{equation*} 
    \lim\limits_{x \rightarrow \infty} \frac{f(tx)}{f(x)} = t^\alpha 
\end{equation*}
holds for all $t \in \R_+$. 
$f$ is said to be slowly varying at $\infty$ if $\alpha$ equals $0$, and rapidly varying at $\infty$ (i.e.\ regularly varying with parameter $\alpha = \infty$) if
\begin{equation*}
    \lim\limits_{x \rightarrow \infty} \frac{f(tx)}{f(x)} = \begin{cases} 0 & t < 1 \\ 1 & t = 1 \\ \infty & t > 1 \end{cases} ~.
\end{equation*}
When no ambiguity is possible, we will call all slowly, regularly and rapidly varying functions regularly varying functions with $\alpha \in \overline{\R}$.

Similarly, we call $f$ regularly varying at 0 with parameter $\alpha \in \overline{\R}$ if $x \mapsto f \rbraces{\frac{1}{x}}$ is regularly varying at $\infty$ with parameter $-\alpha$. 
\end{definition}

We now state an explicit formula for the tail dependence function of an Archimedean copula in terms of its Archimedean generator; the proof goes by direct calculation.

\begin{lemma} \label{lma:regVarTD}
If $C$ is an Archimedean $d$-copula with generator $\phi$ which is regularly varying at $0$ with parameter $-\alpha$, then
\begin{equation*}
    \TDL{\boldm{w}}{C} = \begin{cases}
    						0																	&\text{ if } \alpha = 0				\\										
    						\rbraces{\sum\limits_{k = 1}^d w_k^{-\alpha}}^{-\frac{1}{\alpha}} 	&\text{ if } \alpha \in (0, \infty)\\
    						\min\limits_{k = 1, \ldots, d} w_k									&\text{ if } \alpha = \infty
    			  		\end{cases} ~.
\end{equation*}
\end{lemma} 

In the following we restrict our subsequent analysis to strict generators $\phi$, that is, generators fulfilling 
\begin{equation*}
	\lim\limits_{s \searrow 0} \phi(s) = \infty ~,
\end{equation*}
as the Archimedean copula is otherwise necessarily equal to zero in a neighbourhood near $\boldm{0}$ (see \citet{Charpentier.2009}).
The following proposition gives a reformulation of this fact in our language of stochastic orders. 

\begin{proposition}\label{lma:Charpentier_strict} 
Let $C$ be an Archimedean $d$-copula with nonstrict generator $\phi$ and let $\widetilde{C}$ be an arbitrary $d$-copula. 
Then $C \tdo \widetilde{C}$ is equivalent to $C \loc \widetilde{C}$. 
\end{proposition}

\begin{proof}
Let $C$ be an Archimedean $d$-copula with nonstrict generator $\phi$. 
Following Section~3.2 in \citet{Charpentier.2009}, there exists an $\varepsilon > 0$ such that
\begin{equation*}
    C(\boldm{u}) = 0    \quad \text{ for all } \boldm{u} \in B_\varepsilon(\boldm{0}) \cap [0, 1]^d
\end{equation*}
and necessarily $C \loc \widetilde{C}$ for all $d$-copulas $\widetilde{C}$. 
Moreover, $\TD{\boldm{w}}{C}$ is identically zero for all $\boldm{w} \in \R_+^d$ and therefore $\TD{\boldm{w}}{C} = 0 \leq \TD{\boldm{w}}{\widetilde{C}}$ for all copulas $\widetilde{C}$.
Thus $C \tdo \widetilde{C}$ is equivalent to $C \loc \widetilde{C}$. 
\end{proof}

Note that for strict generators the generalized inverse reduces to the usual inverse.
Let us now state the main theorem before proving the necessary technical results.

\begin{theorem} \label{thm:archimedean_loc_loo}
Let $C_1$ and $C_2$ be Archimedean $d$-copulas with regularly varying generators $\phi_1$ and $\phi_2$, respectively.
Then
\begin{equation*}
	C_1 \stdo C_2 \implies C_1 \loc C_2 ~. 
\end{equation*}
\end{theorem}

We will briefly outline the strategy of the proof, which is given at the end of this section. The proof develops a localized version of \citet[Ch.~4.4]{Nelsen.2006} and consists of the following steps: 
\begin{enumerate}
	\item $C_1 \loc C_2$ whenever $(\phi_1 \circ \Inv{\phi_2})$ is subadditive near $\infty$ (see Lemma~\ref{prop:loc_ult_subadditive}).
	\item $(\phi_1 \circ \Inv{\phi_2})$ is subadditive near $\infty$ if $\frac{\phi_1(x)}{\phi_2(x)}$ is increasing near $0$ (see Lemma~\ref{prop:ndAg}).
	\item $\lambda(C_1) < \lambda(C_2)$ implies that $\frac{\phi_1(x)}{\phi_2(x)}$ is increasing near $0$.
\end{enumerate}

In \citet{Nelsen.2006} the lower orthant ordering of two Archimedean copulas $C_1$ and $C_2$ is characterized in terms of the global subadditivity of $(\phi_1 \circ \Inv{\phi_2})$. 
We can relax this condition and only require subadditivity for large values to ensure the local orthant ordering $\loc$.
Such a localized version has previously also been considered in \citet{Hua.2012} where the generator is assumed to be the Laplace transform of a positive random variable.

\begin{definition}
We say a function $f: \R_+ \rightarrow \R_+$ is subadditive near $\infty$ if there exists an $M \geq 0$ such that 
\begin{equation*}
    f(x + y) \leq f(x) + f(y) 
\end{equation*}
holds for all $x, y \in [M, \infty)$.
\end{definition}

\begin{lemma} \label{prop:loc_ult_subadditive}
Let $C_1$ and $C_2$ be Archimedean $d$-copulas generated by $\phi_1$ and $\phi_2$, respectively. 
Then $C_1 \loc C_2$ whenever $(\phi_1 \circ \Inv{\phi_2})$ is subadditive near $\infty$.
\end{lemma}

\begin{proof}
Consider $f: \R_+ \rightarrow \R_+$ defined by $f(x) := (\phi_1 \circ \Inv{\phi_2})(x)$. 
As $f$ is subadditive near $\infty$, we have
\begin{equation*}
    f \rbraces{\sum\limits_{\ell = 1}^d x_\ell} \leq \sum\limits_{\ell = 1}^d f(x_\ell) \qquad \text{ for all } x_1, \ldots, x_d \geq M ~.
\end{equation*}
Applying the strictly decreasing function $\Inv{\phi_1}$ to the above inequality yields
\begin{equation*} 
    \Inv{\phi_2}\rbraces{\sum\limits_{\ell = 1}^d x_\ell} = \rbraces{\Inv{\phi_1} \circ f} \rbraces{\sum\limits_{\ell = 1}^d x_\ell}
                                                          \geq \Inv{\phi_1}\rbraces{\sum\limits_{\ell = 1}^d f(x_\ell)}   \qquad \text{ for all } x_1, \ldots, x_d \geq M ~.
\end{equation*}
Since $\phi_2$ is strictly decreasing, $x_\ell \geq M$ implies $\Inv{\phi_2}(x_\ell) \leq \Inv{\phi}_2(M)$.
Therefore, for any $\boldm{u} \in B_\varepsilon(\boldm{0}) \cap [0, 1]^d$ with $\varepsilon := \Inv{\phi_2}(M) > 0$
and $x_\ell := \phi_2(u_\ell) \geq \phi_2(\varepsilon) = M$, we have
\begin{align*}
    \Inv{\phi_2}\rbraces{\sum\limits_{\ell = 1}^d \phi_2(u_\ell)} 
                                                    &= \Inv{\phi_2}\rbraces{\sum\limits_{\ell = 1}^d x_\ell}
                                                    \geq \Inv{\phi_1}\rbraces{\sum\limits_{\ell = 1}^d f(x_\ell)} \\
                                                    &= \Inv{\phi_1}\rbraces{\sum\limits_{\ell = 1}^d \rbraces{\phi_1 \circ \Inv{\phi_2}}(\phi_2(u_\ell))} 
                                                    = \Inv{\phi_1} \rbraces{\sum\limits_{\ell = 1}^d \phi_1 (u_\ell)} ~.  
\end{align*}
This yields $C_1(\boldm{u}) \leq C_2(\boldm{u})$ for all $\boldm{u} \in B_\varepsilon(\boldm{0}) \cap [0, 1]^d$ and thus $C_1 \loc C_2$.
\end{proof}

Unfortunately, the subadditivity of $\phi_1 \circ \Inv{\phi_2}$ near $\infty$ is rather difficult to validate and we therefore present a sufficient criterion similar to Corollary~4.4.5 in \citet{Nelsen.2006}.

\begin{lemma} \label{prop:ndAg}
Let $C_1$ and $C_2$ be Archimedean $d$-copulas generated by $\phi_1$ and $\phi_2$, respectively. 
Then $(\phi_1 \circ \Inv{\phi_2})$ is subadditive near $\infty$ whenever $\frac{\phi_1}{\phi_2}$ is increasing on $(0, \varepsilon)$ for some $\varepsilon > 0$.
\end{lemma}

\begin{proof}
Define $g: (0, \infty) \rightarrow (0, \infty)$ as $g(x) := \frac{f(x)}{x}$ and $f(x) := (\phi_1 \circ \Inv{\phi_2})(x)$ as before. 
We now verify that $g$ is decreasing on $(M, \infty)$ with $M := \phi_2(\varepsilon)$.
As $\phi_2$ is strictly decreasing for all $x \geq 0$, we have for $M \leq x \leq y$
\begin{equation*}
     0 < \Inv{\phi}_2(y) \leq \Inv{\phi}_2(x) \leq  \Inv{\phi}_2(M) = \varepsilon ~.
\end{equation*}
Combining this with the fact that $g \circ \phi_2 = \frac{\phi_1}{\phi_2}$ is increasing on $(0, \varepsilon)$, we have
\begin{equation*}
    g(x) = \rbraces{g \circ \phi_2} \rbraces{\Inv{\phi}_2(x)} \geq \rbraces{g \circ \phi_2} \rbraces{\Inv{\phi}_2(y)} = g(y) ~,
\end{equation*}
so $g$ is decreasing for $x, y \geq M$. 
Thus, for all $x, y \in [M, \infty)$
\begin{equation*}
    x \rbraces{g(x + y) - g(x)} + y \rbraces{g(x + y) - g(y)} \leq 0
\end{equation*} 
or equivalently
\begin{equation*}
    f(x + y) = (x + y)g(x + y) \leq xg(x) + yg(y) = f(x) + f(y) ~. 
\end{equation*}
Thus, $f$ is subadditive near $\infty$. 
\end{proof}

Up to this point, our approach was entirely independent of the tail dependence function of the Archimedean copulas.
To ensure the existence of a tail dependence function we will assume that the generators are regularly varying according to Definition~\ref{def:reg_var}.

We will show that ordered tail dependence functions result in ordered parameters of regular variation, allowing us to apply Lemma~\ref{prop:ndAg} for Archimedean copulas with regularly varying generators.

\begin{lemma}\label{lemma:archimedean_regular_varying}
Let $C_1$ and $C_2$ be Archimedean $d$-copulas with strict generators $\phi_1$ and $\phi_2$ which are regularly varying at $0$ with parameters $-\alpha_1$ and $-\alpha_2 \in [-\infty, 0]$, respectively.
Then the following are equivalent: 
\begin{enumerate}
	\item $C_1 \stdo C_2$ \label{lemma:archimedean_regular_varyingA}
	\item $\TDC(C_1)	< \TDC(C_2)$  \label{lemma:archimedean_regular_varyingB}
	\item $\alpha_1 < \alpha_2$. \label{lemma:archimedean_regular_varyingC}
\end{enumerate}
\end{lemma}

\begin{proof}
The first implication \ref{lemma:archimedean_regular_varyingA}$\implies$\ref{lemma:archimedean_regular_varyingB} is immediate.
For \ref{lemma:archimedean_regular_varyingB}$\implies$\ref{lemma:archimedean_regular_varyingC}, we deal with all the possible cases seperately.
\begin{enumerate}
	\item If $\TDC(C_1) = 0$ and $\TDC(C_2) \in (0, 1)$, then
	\begin{equation*}
        0 = \lambda(C_1) < \lambda(C_2) = d^{-\frac{1}{\alpha_2}} 
    \end{equation*}
    and thus $\alpha_2 > 0 = \alpha_1$ holds. 
    \item If both $\TDC(C_1)$ and $\TDC(C_2)$ take values in $(0, 1)$, then $\alpha_1 < \alpha_2$ due to
    \begin{equation*}
    d^{-\frac{1}{\alpha_1}} = \lambda(C_1) < \lambda(C_2) = d^{-\frac{1}{\alpha_2}} 
    \end{equation*}
    \item If $\TDC(C_1) \in (0, 1)$ and $\TDC(C_2) = 1$, then
    \begin{equation*}
        d^{-\frac{1}{\alpha_1}} = \lambda(C_1) < \lambda(C_2) = 1 
     \end{equation*}
    and thus $\alpha_1 < \infty = \alpha_2$. 
    \item Lastly, if $\TDC(C_1) = 0$ and $\TDC(C_2) = 1$, then $a_1 = 0 < \infty = a_2$.  
\end{enumerate}
The last implication \ref{lemma:archimedean_regular_varyingC}$\implies$\ref{lemma:archimedean_regular_varyingA} follows from \citet[Ex.~3.8]{Li.2013}.
\end{proof}

\begin{proof}[Proof of Theorem~\ref{thm:archimedean_loc_loo}]
To show $C_1 \loc C_2$, we will invoke Lemma~\ref{prop:ndAg} and show that $\psi(x) := \frac{\phi_1(x)}{\phi_2(x)}$ is increasing in a neighbourhood of 0.
An application of Lemma~\ref{prop:loc_ult_subadditive} then yields the desired result. 
Due to Lemma~\ref{lemma:archimedean_regular_varying}, we have $\alpha_1 < \alpha_2$, where $\phi_i$ is regularly varying with coefficient $-\alpha_i$ in $0$. 
Furthermore, as $\phi_1$ and $\phi_2$ are positive, convex, and regularly varying functions, \citet[Lemma~A.1]{Charpentier.2009} yields
\begin{equation*}
    \lim\limits_{s \searrow 0} \frac{s \phi_i'(s)}{\phi_i(s)} = -\alpha_i 
\end{equation*} 
for $i = 1,2$.
Here, $\phi_i'$ denotes an increasing representative of the derivative of $\phi_i$, which is only defined almost everywhere.
$\psi$ is almost everywhere differentiable as the ratio of continuous and almost everywhere differentiable functions, where 
\begin{equation*}
	\frac{\phi_1'(s) \phi_2(s) - \phi_1(s) \phi_2'(s)}{\phi_2(s)^2}
\end{equation*}
is a representative of the derivative of $\psi$ which we will denote by $\psi'$. 
This implies
\begin{align*}
    \lim\limits_{s \searrow 0} \frac{s \psi'(s)}{\psi(s)}	
                                                    		&= \lim\limits_{s \searrow 0} \rbraces{s \frac{\phi_1'(s) \phi_2(s) - \phi_1(s) \phi_2'(s)}{\phi_2(s)^2}} \frac{\phi_2(s)}{\phi_1(s)} \\
                                                    		&= \lim\limits_{s \searrow 0} \frac{s \phi_1'(s)}{\phi_1(s)}  - \frac{s\phi_2'(s)}{\phi_2(s)}
                                                    		= -\alpha_1 - (-\alpha_2)
                                                    		= \alpha_2 - \alpha_1 > 0 ~.
\end{align*}

Due to $\psi \geq 0$, $\psi'$ must be positive on $(0, \varepsilon)$ for some $\varepsilon > 0$. 
Finally, given $x_1, x_2 \in (0, \varepsilon)$ with $x_1 \leq x_2$, we have that $\phi_1$ and $\phi_2$ are absolutely continuous on $[x_1, x_2]$. 
This yields that $\psi$ as the ratio of absolutely continuous functions on $[x_1, x_2]$ is absolutely continuous on $[x_1, x_2]$ and therefore increasing on $(0, \varepsilon)$ due to
\begin{equation*}
	\psi(x_2) - \psi(x_1)	= \int\limits_{x_1}^{x_2} \psi'(s) \; \mathrm{d} s ~. \qedhere 
\end{equation*} 
\end{proof}

Finally, we point out that Theorem~\ref{thm:archimedean_loc_loo} also yields an approach to establish ordering results for multivariate copulas using Archimedean $d$-copulas as building blocks; the resulting copulas are called \emph{hierarchical (or nested) Archimedean copulas} (see, e.g., \citet{McNeil.2008} and \citet{Okhrin.2013}).

The following example is meant as an illustration of the general approach by dealing with the simplest possible hierarchical Archimedean copula. 

\tikzset{
    >=stealth',
    punkt/.style={
           rectangle,
           rounded corners,
           draw=black, thick,
           text width=2em,
           minimum height=2em,
           text centered},
    pil/.style={
           <-,
           thick,
           shorten <=2pt,
           shorten >=2pt,}
}

\begin{example}
Consider the hierarchical Archimedean $3$-copulas 
\begin{equation*}
    C(u_1, u_2, u_3) := C_1(u_1, C_2(u_2, u_3)) \text{\quad and \quad} D(u_1, u_2, u_3) := D_1(u_1, D_2(u_2, u_3))
\end{equation*}
where $C_i$ and $D_i$ are Archimedean $2$-copulas with regularly varying generators.
The structure of the hierarchical Archimedean copulas $C$ and $D$ is visualized in Figure~\ref{fig:nested-archimedean} which also suggests an intuitive way to establish our local stochastic ordering, namely by comparing the respective nodes $C_i$ and $D_i$ of $C$ and $D$ separately.
\begin{figure}
\begin{center}
\begin{subfigure}[b]{0.5\textwidth}
\centering
\begin{tikzpicture}[node distance=1cm, auto, punkt]
\node[punkt] (copula) {$C_1$};
\node[punkt, below left=of copula] (var1) {$u_1$}
edge[pil] (copula.south west);
\node[punkt, below right=of copula] (nest1) {$C_2$}
edge[pil] (copula.south east);
\node[punkt, below left=of nest1] (var2) {$u_2$}
edge[pil] (nest1.south west);
\node[punkt, below right=of nest1] (nest2) {$u_3$}
edge[pil] (nest1.south east);
\end{tikzpicture}
\end{subfigure}%
\begin{subfigure}[b]{0.5\textwidth}
\centering
\begin{tikzpicture}[node distance=1cm, auto, punkt]
\node[punkt] (copula) {$D_1$};
\node[punkt, below left=of copula] (var1) {$u_1$}
edge[pil] (copula.south west);
\node[punkt, below right=of copula] (nest1) {$D_2$}
edge[pil] (copula.south east);
\node[punkt, below left=of nest1] (var2) {$u_2$}
edge[pil] (nest1.south west);
\node[punkt, below right=of nest1] (nest2) {$u_3$}
edge[pil] (nest1.south east);
\end{tikzpicture}
\end{subfigure}
\caption{The hierarchical Archimedean $3$-copulas $C$ and $D$.}
\label{fig:nested-archimedean}
\end{center}
\end{figure}
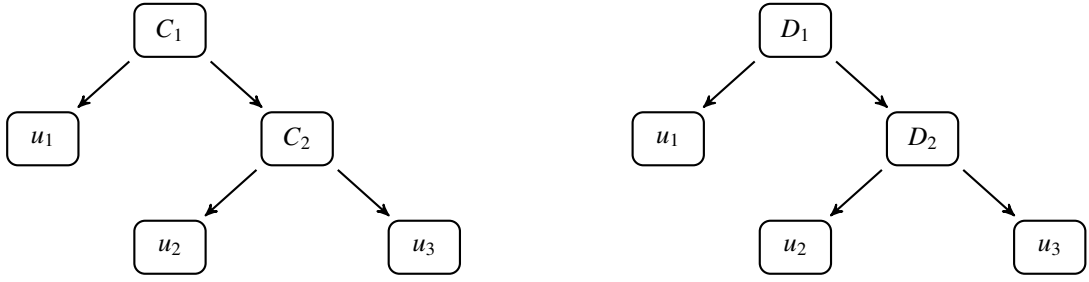
For this let us assume that
\begin{equation} \label{eqn:hac_ordering}
    \lim\limits_{w_i \rightarrow \infty} \TDL{\boldm{w}}{C} < \lim\limits_{w_i \rightarrow \infty} \TDL{\boldm{w}}{D}
\end{equation}
for all $w = (w_1, w_2, w_3) \in \mathbb{R}_+^3$ and $i\in \{ 1,2,3 \}$, which is just a slightly stronger assumption than our usual condition $C \stdo D$.  We claim that $$ C \loc D ~. $$
For this, we use the representation of $\TDL{\wc}{C_i}$ and $\TDL{\wc}{D_i}$ given in Lemma~\ref{lma:regVarTD} and find
\begin{equation*}
    \lim\limits_{w_1 \rightarrow \infty} \TDL{(w_1, w_2, w_3)}{C} = \TDL{(w_2, w_3)}{C_2}  
    \text{ and }
    \lim\limits_{w_1 \rightarrow \infty} \TDL{(w_1, w_2, w_3)}{D} = \TDL{(w_2, w_3)}{D_2} ~.
\end{equation*}
Combining \eqref{eqn:hac_ordering} with Theorem~\ref{thm:archimedean_loc_loo} yields $C_2 \loc D_2$.
Similarly, we obtain $C_1 \loc D_1$. Then, using the monotonicity of $v \mapsto C_1(u, v)$ for all $u$ we conclude that
\begin{align*}
    C(u_1, u_2, u_3)    &= C_1(u_1, C_2(u_2, u_3)) 
                        \leq C_1(u_1, D_2(u_2, u_3)) \\
                        &\leq D_1(u_1, D_2(u_2, u_3)) 
                        = D(u_1, u_2, u_3) 
\end{align*}
for all $(u_1, u_2, u_3)$ near $\bm{0}$, so that indeed we have $C \loc D$.
\end{example}

\section{Implications for risk management\label{sec:6}}

In addition to the theoretical results presented in the previous sections, we are going to illustrate their relevance by pointing out an implication to the study of extremal events. Consider, for instance, the returns of a portfolio comprised of various financial assets, or the water levels of different reservoirs (or rivers). Extremal events in these situations describe large joint losses in the portfolio or dangerously low (or high) water levels at the same time.

Traditionally, the probability of extremal events is measured by the tail dependence coefficient $\lambda(\boldm{X})$ where $\boldm{X}$ is the random vector consisting of the corresponding returns or water levels, respectively. A low tail dependence coefficient indicates a low probability of such worst-case scenarios, and that has always been interpreted as indicating low risk as well. However, as heuristically convincing as this connection may seem, there is only little evidence substantiating it (see, e.g., \citet{Embrechts.2005, Kaas.2009, Rueschendorf.2012}).

It is crucial to note in this context that tail dependence is defined in terms of the directional \emph{derivative} of the copula at zero, whereas the classical measures of risk---such as Value-at-Risk or Expected Shortfall---depend on the quantiles of the joint distribution function, i.e., the \emph{values} of the underlying copula. In view of Example~\ref{example:eo_doesnotimply_loo}, however, these two notions do not imply each other, and it is not possible in general to infer a low(er) risk from a low(er) tail dependence since the two orders $\tdo$ and $\loc$ are not equivalent.

A link between the values and the tail dependence coefficient has been studied in \citet{Schmid.2007} where it is shown that a localized version of Spearman's $\rho$, defined as $$ \rho(p \, ; C) := a(p) \int\limits_{[0,p]^d} C(\bm{u}) \, \textrm{d}\bm{u} + b(p) $$ with normalizing constants $a(p)$ and $b(p)$, satisfies $$ \TDC(C) \leq \lim\limits_{p\to 0} \rho(p \, ; C) = (d+1) \int\limits_{[0,1]^d} \TDL{\bm{w}}{C} \textrm{d}\bm{w} $$ if the limits exists. But, being a limiting result, this cannot be used to derive any risk estimates in terms of the tail dependence function.

However, in view of our results in Section~\ref{section:classes_tail_behaviour}, we finally are able to conclude that a smaller tail dependence function implies a smaller risk of extremal events, at least for all the classes of copulas discussed in Section~\ref{section:classes_tail_behaviour}. Of particular importance for practical applications are the Archimedean copulas since, e.g., many higher dimensional models in financial mathematics are built upon such copulas, and it is common belief that Archimedean copulas provide particularly good models for financial data (see \citet{Buecher.2012} and \citet{Trede.2013}).

\section*{Acknowledgements}

We thank the two anonymous referees for their constructive comments and enlightening remarks which helped us to improve the article.
The second author gratefully acknowledges financial support from the German Academic Scholarship Foundation.

\bibliographystyle{myjmva}
\bibliography{reference}

\end{document}